\documentclass[12pt,a4paper]{amsart}
\usepackage{amssymb,amsmath,amsthm}
\usepackage{graphicx}
\usepackage[colorlinks=true, allcolors=blue]{hyperref}
\usepackage{enumitem}

\usepackage{graphicx}        
\usepackage{tikz}
\usetikzlibrary{backgrounds}

\newcommand\cC{{\mathcal C}}

\newcommand\cF{{\mathcal F}}

\newcommand{\floor}[1]{\left\lfloor{#1}\right\rfloor}
\newcommand{\ceil}[1]{\left\lceil{#1}\right\rceil}

\newcommand{\nchn}{\binom{n}{\lfloor n/2\rfloor}}

\newcommand\bC{\mathbf C}

\newcommand\cH{{\mathcal H}}

\newcommand\cP{{\mathcal P}}

\newcommand\bN{\textbf{N}}

\newcommand\bfly{\bowtie}
\newcommand\La{{\rm La}}

\theoremstyle{plain}
\newtheorem{theorem}{Theorem}
\newtheorem{lemma}[theorem]{Lemma}

\newtheorem{problem}[theorem]{Problem}

\theoremstyle{definition}

\newtheorem{defn}[theorem]{Definition}

\newtheorem{remark}[theorem]{Remark}
\newtheorem{rem}[theorem]{Remark}

\newcommand\tref[1]{Theorem~\ref{thm:#1}}
\newcommand\cref[1]{Corollary~\ref{cor:#1}}

\title{On generalized Tur\'an results in height two posets}

\author{J\'ozsef Balogh}
\address{University of Illinois at Urbana-Champaign and Moscow Institute of Physics and Technology}
\email{jobal@illinois.edu}
\thanks{Balogh's research is partially
supported by NSF grants DMS-1764123 and RTG DMS-1937241, the Arnold O. Beckman Research Award (UIUC
Campus Research Board RB 18132), the Langan Scholar Fund (UIUC), and the Simons Fellowship.}

\author{Ryan R. Martin}
\address{Iowa State University}
\email{rymartin@iastate.edu}
\thanks{Martin's research is partially supported by Simons Foundation Collaboration grants \#353292 and \#709641.}

\author{D\'aniel T. Nagy}
\address{Alfr\'ed R\'enyi Institute of Mathematics}
\email{nagydani@renyi.hu}
\thanks{Nagy's research is partially supported by NKFIH grants FK 132060, K 132696 and PD 137779 and by the J\'anos Bolyai Research Fellowship of the Hungarian Academy of Sciences}

\author{Bal\'azs Patk\'os}
\address{Alfr\'ed R\'enyi Institute of Mathematics and Moscow Institute of Physics and Technology}
\email{patkos@renyi.hu}
\thanks{Patk\'os's research is partially supported by NKFIH grants SNN 129364 and FK 132060.}

\date{}

\begin{document}

\maketitle

\begin{abstract}
	For given posets $P$ and $Q$ and an integer $n$, the generalized Tur\'an problem for posets, asks for the maximum number of copies of $Q$ in a $P$-free subset of the $n$-dimensional Boolean lattice, $2^{[n]}$. 
    
	In this paper, among other results, we show the following:
	\begin{itemize}
		\item[(i)] 	For every  $n\geq 5$, the maximum number of $2$-chains in a butterfly-free subfamily of  $2^{[n]}$ is $\left\lceil\frac{n}{2}\right\rceil\binom{n}{\lfloor n/2\rfloor}$. 
		\item[(ii)] 	For every  fixed $s$, $t$ and $k$, a $K_{s,t}$-free family in $2^{[n]}$ has $O\left(n\binom{n}{\lfloor n/2\rfloor}\right)$ $k$-chains.
		\item[(iii)] 	For every $n\geq 3$, the maximum number of $2$-chains in an $\textbf{N}$-free family is $\binom{n}{\lfloor n/2\rfloor}$, where $\textbf{N}$ is a poset  on 4 distinct elements $\{p_1,p_2,q_1,q_2\}$ for which $p_1 < q_1$, $p_2 < q_1$ and $p_2 < q_2$.
		\item[(iv)] 	We also prove exact results for the maximum number of $2$-chains in a family that has no $5$-path and asymptotic estimates for the number of $2$-chains in a family with no $6$-path. 
	\end{itemize}
\end{abstract}

\section{Introduction}

We say that $P$ is a \textit{(weak) subposet} of $Q$ if there is an injection $f: P\rightarrow Q$ such that $p\leq_P p'$ implies $f(p)\leq_Q f(p')$. 
Denote $P_k$  the chain with $k$ elements.
For positive integers $s$ and $t$, the poset $K_{s,t}$ has $s$ minimal elements, $t$ maximal elements and every minimal element is less than every maximal element. We write   $\bfly$ for  $K_{2,2}$, which is also called the {\it butterfly}.

For integers $n$ and $k$, let $[n]=\{1,\ldots,n\}$ and let 
\begin{align*}
    \binom{[n]}{k} = \left\{S\subseteq [n] : |S|=k\right\} .
\end{align*}

 We will sometimes use $abc$ to denote the set $\{a,b,c\}$. 
    In the Boolean lattice of dimension $n$, a {\it full chain} $\mathcal{C}$ is a sequence of $n+1$ sets of the form $\emptyset = C_0 \subset C_1 \subset\cdots\subset C_{n-1} \subset C_n = [n]$.

The expression $\La(n,P)$ denotes the largest subposet of the $n$-dimensional Boolean lattice that does not have $P$ as a subposet.
The study of $\La(n,P)$ can be traced to Sperner~\cite{S1928}, who proved that $\La(n,P_2)=\nchn$. 
This was extended by Erd\H{o}s~\cite{E1945} in 1945, who established $\La(n,P_{k+1})\sim k\nchn$, and in fact determined $\La(n,P_{k+1})$ exactly. 
The systematic study of $\La(n,P)$ for general $P$ was initiated by  Katona and Tarj\'an~\cite{KT} in 1983, who studied the case where $P$ is a fork poset, in which the {\it $r$-fork} poset on $r+1$ elements  is the poset consisting of $a,b_1,b_2,\dots,b_r$ with $a$ being smaller than all $b_i$s and the $b_i$s forming an antichain. 

In 2005, De Bonis, Katona, and Swanepoel~\cite{DKS} proved that for all $n\geq 3$, $\La(n,\bfly) = \binom{n}{\lfloor n/2\rfloor} + \binom{n}{\lfloor n/2\rfloor+1}$.


We emphasize that the  {\it butterfly} poset is a set of 4 distinct elements $\{p_1,p_2,q_1,q_2\}$ for which $p_1<q_1$, $p_1<q_2$, $p_2<q_1$, and $p_2<q_2$. 
Note that this also forbids a 4-element poset for which both $p_1 < p_2 < q_1$ and $p_1 < p_2 < q_2$ or both $p_1 < q_1 < q_2$ and $p_2 < q_1 < q_2$ or even $p_1 < p_2 < q_1 < q_2$, see Figure~\ref{fig:butterfly}.
\begin{figure}[ht]
    \begin{center}
        \pgfdeclarelayer{bg}
        \pgfsetlayers{bg,main}
        \tikzset{vertex/.style={circle, draw=white, thin, fill=black, minimum size = 4pt},
        edge/.style={black, ultra thick}}
        \def\shft{13pt}
        \begin{tikzpicture}
	        \begin{scope}
		        \node [vertex] (p1) at (-3/4,-3/4) {};
    		    \node [vertex] (p2) at (3/4,-3/4) {};
	    	    \node [vertex] (q1) at (-3/4,3/4) {};
		        \node [vertex] (q2) at (3/4,3/4) {};
        		\node [yshift=-\shft] at (p1)   {$p_1$};
        		\node [yshift=-\shft] at (p2) {$p_2$};
	        	\node [yshift=\shft] at (q1) {$q_1$};
	    	    \node [yshift=\shft] at (q2) {$q_2$};
        		\begin{pgfonlayer}{bg}
	    	        \draw [edge] (p1) -- (q1);
		            \draw [edge] (p1) -- (q2);
		            \draw [edge] (p2) -- (q1);
        		    \draw [edge] (p2) -- (q2);
    	    	\end{pgfonlayer}
    	    \end{scope}
	
        	\begin{scope}[xshift=100pt]
    	    	\node [vertex] (p1) at (0,-1) {};
    		    \node [vertex] (p2) at (0,0) {};
        		\node [vertex] (q1) at (-3/4,1) {};
    	    	\node [vertex] (q2) at (3/4,1) {};
        		\node [yshift=-\shft] at (p1)   {$p_1$};
    		    \node [xshift=\shft] at (p2) {$p_2$};
	    	    \node [yshift=\shft] at (q1) {$q_1$};
		        \node [yshift=\shft] at (q2) {$q_2$};
    	    	\begin{pgfonlayer}{bg}
	        	    \draw [edge] (p1) -- (p2);
    		        \draw [edge] (p2) -- (q1);
		            \draw [edge] (p2) -- (q2);
    		    \end{pgfonlayer}
	        \end{scope}
	
        	\begin{scope}[xshift=195pt]
		        \node [vertex] (p1) at (-3/4,-1) {};
        		\node [vertex] (p2) at (3/4,-1) {};
		        \node [vertex] (q1) at (0,0) {};
        		\node [vertex] (q2) at (0,1) {};
    		    \node [yshift=-\shft] at (p1)   {$p_1$};
    	    	\node [yshift=-\shft] at (p2) {$p_2$};
	        	\node [xshift=\shft] at (q1) {$q_1$};
    		    \node [yshift=\shft] at (q2) {$q_2$};
		        \begin{pgfonlayer}{bg}
        		    \draw [edge] (p1) -- (q1);
		            \draw [edge] (p2) -- (q1);
        		    \draw [edge] (q1) -- (q2);
		        \end{pgfonlayer}
        	\end{scope}
	
        	\begin{scope}[xshift=280pt]
		        \node [vertex] (p1) at (0,-3/2) {};
        		\node [vertex] (p2) at (0,-1/2) {};
		        \node [vertex] (q1) at (0,1/2) {};
        		\node [vertex] (q2) at (0,3/2) {};
    		    \node [xshift=\shft] at (p1)   {$p_1$};
    	    	\node [xshift=\shft] at (p2) {$p_2$};
	        	\node [xshift=\shft] at (q1) {$q_1$};
    		    \node [xshift=\shft] at (q2) {$q_2$};
		        \begin{pgfonlayer}{bg}
        		    \draw [edge] (p1) -- (p2);
		            \draw [edge] (p2) -- (q1);
        		    \draw [edge] (q1) -- (q2);
        		\end{pgfonlayer}
	        \end{scope}
	    \end{tikzpicture}
        \caption{Forbidden configurations of a $\bfly$-free family.}
        \label{fig:butterfly}
    \end{center}
\end{figure}


The $\bN$ poset is a set of 4 distinct elements $\{p_1,p_2,q_1,q_2\}$ for which $p_1 < q_1$, $p_2 < q_1$ and $p_2 < q_2$. 
Note that this also forbids a $\bfly$, see Figure~\ref{fig:Nposet}.

\begin{figure}[ht]
    \begin{center}
        \pgfdeclarelayer{bg}
        \pgfsetlayers{bg,main}
        \tikzset{vertex/.style={circle, draw=white, thin, fill=black, minimum size = 4pt},
        edge/.style={black, ultra thick}}
        \def\shft{13pt}
        \begin{tikzpicture}
	        \begin{scope}
		        \node [vertex] (p1) at (-3/4,-3/4) {};
    		    \node [vertex] (p2) at (3/4,-3/4) {};
	    	    \node [vertex] (q1) at (-3/4,3/4) {};
		        \node [vertex] (q2) at (3/4,3/4) {};
        		\node [yshift=-\shft] at (p1)   {$p_1$};
        		\node [yshift=-\shft] at (p2) {$p_2$};
	        	\node [yshift=\shft] at (q1) {$q_1$};
	    	    \node [yshift=\shft] at (q2) {$q_2$};
        		\begin{pgfonlayer}{bg}
	    	        \draw [edge] (p1) -- (q1);
		            \draw [edge] (p2) -- (q1);
		            \draw [edge] (p2) -- (q2);
    	    	\end{pgfonlayer}
    	    \end{scope}
	
	        \begin{scope}[xshift=80pt]
		        \node [vertex] (p1) at (-3/4,-3/4) {};
    		    \node [vertex] (p2) at (3/4,-3/4) {};
	    	    \node [vertex] (q1) at (-3/4,3/4) {};
		        \node [vertex] (q2) at (3/4,3/4) {};
		        \node [yshift=-\shft] at (p1)   {$p_1$};
       		\node [yshift=-\shft] at (p2) {$p_2$};
	        	\node [yshift=\shft] at (q1) {$q_1$};
	    	    \node [yshift=\shft] at (q2) {$q_2$};
        		\begin{pgfonlayer}{bg}
	    	        \draw [edge] (p1) -- (q1);
		            \draw [edge] (p1) -- (q2);
		            \draw [edge] (p2) -- (q1);
		            \draw [edge] (p2) -- (q2);
    	    	\end{pgfonlayer}
    	    \end{scope}
	
        	\begin{scope}[xshift=160pt]
    	    	\node [vertex] (p1) at (0,-1) {};
    		    \node [vertex] (p2) at (0,0) {};
        		\node [vertex] (q1) at (-3/4,1) {};
    	    	\node [vertex] (q2) at (3/4,1) {};
        		\node [yshift=-\shft] at (p1)   {$p_1$};
   		    \node [xshift=\shft] at (p2) {$p_2$};
	    	    \node [yshift=\shft] at (q1) {$q_1$};
		        \node [yshift=\shft] at (q2) {$q_2$};
    	    	\begin{pgfonlayer}{bg}
	        	    \draw [edge] (p1) -- (p2);
    		        \draw [edge] (p2) -- (q1);
		            \draw [edge] (p2) -- (q2);
    		    \end{pgfonlayer}
	        \end{scope}
	
        	\begin{scope}[xshift=230pt]
		        \node [vertex] (p1) at (-3/4,-1) {};
        		\node [vertex] (p2) at (3/4,-1) {};
		        \node [vertex] (q1) at (0,0) {};
        		\node [vertex] (q2) at (0,1) {};
    		    \node [yshift=-\shft] at (p1)   {$p_1$};
    	    	\node [yshift=-\shft] at (p2) {$p_2$};
	        	\node [xshift=\shft] at (q1) {$q_1$};
   		    \node [yshift=\shft] at (q2) {$q_2$};
		        \begin{pgfonlayer}{bg}
        		    \draw [edge] (p1) -- (q1);
		            \draw [edge] (p2) -- (q1);
        		    \draw [edge] (q1) -- (q2);
		        \end{pgfonlayer}
        	\end{scope}
	
        	\begin{scope}[xshift=290pt]
		        \node [vertex] (p1) at (0,-3/2) {};
        		\node [vertex] (p2) at (0,-1/2) {};
		        \node [vertex] (q1) at (0,1/2) {};
        		\node [vertex] (q2) at (0,3/2) {};
    		    \node [xshift=\shft] at (p1)   {$p_1$};
    	    	\node [xshift=\shft] at (p2) {$p_2$};
	        	\node [xshift=\shft] at (q1) {$q_1$};
    		    \node [xshift=\shft] at (q2) {$q_2$};
		        \begin{pgfonlayer}{bg}
        		    \draw [edge] (p1) -- (p2);
		            \draw [edge] (p2) -- (q1);
        		    \draw [edge] (q1) -- (q2);
        		\end{pgfonlayer}
	        \end{scope}
	    \end{tikzpicture}
        \caption{Forbidden configurations of an $\bN$-free family.}
        \label{fig:Nposet}
    \end{center}
\end{figure}

Determining $\La(n,P)$ is known as the \textit{poset Tur\'an problem}. The basic results of Tur\'an theory in graphs are well-known~\cite{T1941,ES1946,ES1966}, whereas the asymptotic value of $\La(n,P)$ is not known for most posets $P$. Famously, even if $P=\Diamond$, the 4-element diamond poset, the asymptotic value of $\La(n,\Diamond)$ is unknown~\cite{GMT}. 

\subsection{Previous Results}

One extension of Tur\'an theory is when instead of determining the maximum number of edges in an $H$-free graph, one finds the maximum number of copies of $F$ in an $H$-free graph, see for example  Alon and Shikhelman~\cite{AS}.
 
There is an obvious analogue  of generalized Tur\'an theory to posets, introduced by Gerbner, Keszegh, and Patk\'os~\cite{GKP}:
\begin{defn}
    The maximum number of copies of $Q$ in a $P$-free subfamily of the $n$-dimensional Boolean lattice is denoted by $\La(n,P,\#Q)$. 
\end{defn}
We note that the notation is slightly different, but we add the ``$\#$'' symbol to make it clear which poset is being counted.

Denote by $P_k$  the chain on $k$ elements.
Clearly, $\La(n,P)=\La(n,P,\#P_1)$.
The primary problem in  generalized Tur\'an theory on posets is determining $\La(n,P,\#P_{\ell})$, where usually, $\ell$ is less than the height of $P$. 

The function  $\La(n,P_3,\#P_2)$ was determined by Katona~\cite{K1973} and reproved independently in~\cite{P2009}.
The function  $\La(n,P_k,\#P_\ell) $ was determined for every  pair of integers $k > \ell \geq 1$.

\begin{theorem}[Gerbner-Patk\'{o}s~\cite{GP2008}] \label{thm:GP}
	For any pair of integers $k > \ell \geq 1$,
	\begin{align*}
		\La(n,P_k,\#P_\ell) 	= 	\max_{0\leq i_1<i_2<\cdots<i_{k-1}\leq n} f(n,\ell,i_1,i_2,\dots,i_{k-1}),
	\end{align*}
	where $f(n,\ell,i_1,i_2,\dots,i_{k-1})$ denotes the number of $\ell$-chains in $\binom{[n]}{i_1}\cup \binom{[n]}{i_2}\cup \dots \cup \binom{[n]}{i_{k-1}}$.
	If $k=\ell+1$, then $f(n,\ell,i_1,i_2,\dots,i_{k-1})=\binom{n}{n-i_{k-1},i_{k-1}-i_{k-2},\dots,i_2-i_1,i_1}$ and the above maximum is attained when the integers $i_1,i_2-i_1,\ldots,i_{k-1}-i_{k-2},n-i_{k-1}$ differ by at most one. Consequently,
	\begin{align*}
		\La(n,P_3,\#P_2) 		= 	(1+o(1)) \frac{3\sqrt{3}}{2 \pi n} 3^n .
	\end{align*}	
\end{theorem}

The general case of $P$, where $P$ has height at least $3$, has also been resolved. 

\begin{theorem}\label{t13}[Gerbner-Methuku-Nagy-Patk\'{o}s-Vizer~\cite{GMNPV}]~\\
	\begin{enumerate}[label=(\alph*)]
		\item 	If $P$ is a poset of height at least $3$, then
				\begin{align*}
					\La(n,P,\#P_2) 	= 	\Theta\left(\La(n,P_3,\#P_2)\right) .
				\end{align*} \label{it:height3}
		\item 	If $T$ is a poset $T$ of height $2$ whose Hasse diagram is a tree, then 
				\begin{align*}
					\La(n,T,\#P_2) 	= 	\Theta\left(\nchn\right) .
				\end{align*} \label{it:tree}
		\item 	If $P$ is a poset of height $2$ that has at least three elements, then
				\begin{align*}
				    \Omega\left(\nchn\right) 	= 	\La(n,P,\#P_2) 	= 	O\left(n\cdot 2^n\right) .
				\end{align*} \label{it:height2}
	\end{enumerate} \label{thm:GMNPV}
\end{theorem}

Since $\nchn=\Theta\left(2^n/\sqrt{n}\right)$, there is a gap in the asymptotic value of $\La(n,P,\#P_2)$ for height $2$ posets $P$.

\subsection{New results}

We compute $\La(n,\bfly,\#P_2)$, and determine the extremal families.  
\begin{theorem} For all $n\geq 5$,
	\begin{align*}
		\La(n,\bfly,\#P_2) 	=	\left\lceil\frac{n}{2}\right\rceil \nchn . 
	\end{align*} 
	Moreover, if $n\geq 7$, then equality is only achieved by either $\binom{[n]}{\lfloor n/2\rfloor}\cup \binom{[n]}{\lfloor n/2\rfloor+1}$ or $\binom{[n]}{\lceil n/2\rceil-1}\cup \binom{[n]}{\lceil n/2\rceil}$. \label{thm:K22}
\end{theorem}

\begin{rem}
    For $n\leq 4$, the expression in Theorem~\ref{thm:K22} is false. 
    For $n=2$, it is clear that $\La(2, \bfly, \#P_2)=4>2=\binom{2}{1}$. 
    For $n=3$, the family $\{\emptyset,1,2,3,123\}$ witnesses that $\La(3, \bfly, \#P_2)\geq 7>6=2\binom{3}{2}$.
    For $n=4$, the family $\{\emptyset,12,13,14,23,24,34,1234\}$ witnesses that $\La(4, \bfly, \#P_2)\geq 13>12=2\binom{4}{2}$. 

	It is not difficult to see that $\La(3, \bfly, \#P_2)=7$ but we did not choose to go through the case analysis to compute $\La(4, \bfly, \#P_2)$ or to determine all extremal families for $\La(n, \bfly, \#P_2)$ when $n\leq 6$.
\label{rem:smalln}
\end{rem}

\begin{theorem}	\label{thm:Kst}
    Fix $s\geq 2$, $t\geq 2$ and $k\geq 1$. Then, 
	\begin{align*}\textstyle
		\La(n,K_{s,t},\#P_k) 	= O\left(n\cdot \nchn\right) 	= 	O\left(\sqrt{n}\cdot2^n\right) . 
	\end{align*} 
\end{theorem}

Because every finite  poset $P$  of  height $2$ is a subposet of $K_{s,t}$ for some $s$ and $t$, we have an upper bound for $\La(n,P,\#P_2)$ of $O\left(n\cdot \nchn\right)$, which verifies Conjecture 1.6 in~\cite{GMNPV}, and improves the upper bound in Theorem~\ref{t13} (c).

If $\cF=\binom{[n]}{\lfloor n/2\rfloor}\cup \binom{[n]}{\lfloor n/2\rfloor+1}$ does not contain some poset $P$ (in particular, if $P$ contains a $\bfly$), then $\La(n,P,\#P_2)\geq\left\lceil\frac{n}{2}\right\rceil\nchn$.  
Thus, for any poset $P$, the order of magnitude of $\La(n,P,\#P_2)$ is determined if $P$ has height at least $3$ (Theorem~\ref{thm:GP}) and if $P$ has height $2$ and contains a $\bfly$ as a subposet (Theorem~\ref{thm:Kst}).

For other posets, such as the $\ell$-{\it crown}, which is a poset with $\ell$ maximal elements and $\ell$ minimal elements whose Hasse diagram is the $2\ell$-cycle, the order of magnitude of $\La(n,P,\#P_2)$ is unknown, but it is at most $\La(n,K_{\ell,\ell},\#P_2)=O\left(\sqrt{n}\cdot 2^n\right)$. \\

Our other results relate to posets whose Hasse diagram is a tree and the order of magnitude of $\La(n,T,\#P_2)$ is established by Theorem~\ref{thm:GMNPV}\ref{it:tree}.
In our case, we determine the constant coefficient in the case where several tree posets are excluded.
The simplest tree posets are ``path-like'' posets.
\begin{defn}
    Denote $\cP_k$ to be the family of those posets $P$ on $k$ elements such that the undirected Hasse diagram of $P$ is a path.
\end{defn}
Obviously,  $P_k\in \cP_k$ for all $k$, and $\bN \in \cP_4$.
Recall that  $\vee_r$ denotes the \emph{$r$-fork} poset on $r+1$ elements; that is, the poset consisting of $a,b_1,b_2,\dots,b_r$ with $a$ being smaller than all $b_i$s and the $b_i$s forming an antichain. 
Then, $\vee_2\in\cP_3$. Observe that if $k$ is even, then there exists a unique poset in $\cP_k$ of height 2, while if $k$ is odd, then there are two such posets: one being the dual of the other, that is, we can obtain one of them by reversing all relations of the other.

Our final results count $2$-chains in families avoiding all posets of $\cP_k$ or only height-$2$ posets of $\cP_k$, which behaves very differently. 
Note that for $k=4$, there is no difference between these two problems, because the unique height-$2$ tree poset on $4$ elements is $\bN$ and a copy of any tree poset on $4$ elements is a copy of $\bN$.

\begin{theorem}
    For every $n\ge 3$, 
    \begin{align*}\textstyle
        \La(n,{\normalfont \bN},\#P_2)=\La(n,\cP_4,\#P_2)=\nchn .
    \end{align*}\label{thm:N}
\end{theorem}

Denote  by $W$ the poset on $5$ elements $a,b,c,d,e$ with $b < a$, $b < c$, $d < c$, and $d < e$ being all of its relations.
Let $M$ be its dual poset obtained from $W$ by reversing all of its relations,
see Figure~\ref{fig:WMposet}.

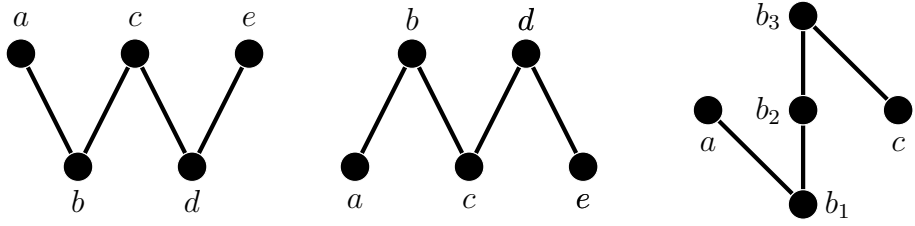
\begin{figure}[ht]
    \begin{center}
        \pgfdeclarelayer{bg}
        \pgfsetlayers{bg,main}
        \tikzset{vertex/.style={circle, draw=white, thin, fill=black, minimum size = 4pt},
        edge/.style={black, ultra thick}}
        \def\shft{13pt}
        \begin{tikzpicture}
	        \begin{scope}
		        \node [vertex] (a) at (-3/2,3/4) {};
    		    \node [vertex] (b) at (-3/4,-3/4) {};
	    	    \node [vertex] (c) at (0,3/4) {};
		        \node [vertex] (d) at (3/4,-3/4) {};
		        \node [vertex] (e) at (3/2,3/4) {};
        		\node [yshift=\shft] at (a)   {$a$};
        		\node [yshift=-\shft] at (b) {$b$};
	        	\node [yshift=\shft] at (c) {$c$};
	    	    \node [yshift=-\shft] at (d) {$d$};
	    	    \node [yshift=\shft] at (e) {$e$};
        		\begin{pgfonlayer}{bg}
	    	        \draw [edge] (a) -- (b) -- (c) -- (d) -- (e);
    	    	\end{pgfonlayer}
    	    \end{scope}
	
	        \begin{scope}[xshift=125pt]
	           \node [vertex] (a) at (-3/2,-3/4) {};
    		    \node [vertex] (b) at (-3/4,3/4) {};
	    	    \node [vertex] (c) at (0,-3/4) {};
		        \node [vertex] (d) at (3/4,3/4) {};
		        \node [vertex] (e) at (3/2,-3/4) {};
        		\node [yshift=-\shft] at (a)   {$a$};
        		\node [yshift=\shft] at (b) {$b$};
	        	\node [yshift=-\shft] at (c) {$c$};
	    	    \node [yshift=\shft] at (d) {$d$};
	    	    \node [yshift=-\shft] at (e) {$e$};
        		\begin{pgfonlayer}{bg}
	    	        \draw [edge] (a) -- (b) -- (c) -- (d) -- (e);
    	    	\end{pgfonlayer}
    	    \end{scope}
    	    
    	   \begin{scope}[xshift=250pt]
	           \node [vertex] (a) at (-5/4,0) {};
    		    \node [vertex] (b1) at (0,-5/4) {};
	    	    \node [vertex] (b2) at (0,0) {};
		        \node [vertex] (b3) at (0,5/4) {};
		        \node [vertex] (c) at (5/4,0) {};
        		\node [yshift=-\shft] at (a)   {$a$};
        		\node [xshift=\shft] at (b1) {$b_1$};
	        	\node [xshift=-\shft] at (b2) {$b_2$};
	        	\node [xshift=-\shft] at (b3) {$b_3$};
	        	\node [yshift=-\shft] at (c) {$c$};
	    	    \node [yshift=\shft] at (d) {$d$};
	    	    \node [yshift=-\shft] at (e) {$e$};
        		\begin{pgfonlayer}{bg}
	    	        \draw [edge] (a) -- (b1) -- (b2) -- (b3) -- (c);
    	    	\end{pgfonlayer}
    	    \end{scope}
	    \end{tikzpicture}
        \caption{The $W$ poset, the $M$ poset, and the $S$ poset.}
        \label{fig:WMposet}
    \end{center}
\end{figure}

Observe that $W$ and $M$ are the unique posets of $\cP_5$ of height $2$, but this time a copy of the poset $S\in \cP_5$ is not necessarily a copy of either $W$ or $M$, where the relations of $S$ on $a,b_1,b_2,b_3,c$ are $b_1 < a$, $b_1 < b_2 < b_3$, and $c < b_3$, see Figure~\ref{fig:WMposet}. 
It turns out that forbidding all of $\cP_5$ and forbidding only $W$ and $M$ result in completely different extremal values and families as seen in Theorems~\ref{thm:P5} and~\ref{thm:P6} below.

\begin{theorem}
    If $n$ is sufficiently large, then
    \begin{align*} \textstyle
        \La(n,\cP_5,\#P_2)
        =5\cdot\binom{n-2}{\floor{n/2}-1}
        =\left(\frac{5}{4}+o(1)\right)\cdot
        \nchn .
    \end{align*} \label{thm:P5}
\end{theorem}

\begin{theorem}
    If $n\geq 2$ then 
    \begin{align*} \textstyle
        2\nchn+1
        \le \La(n,\{W,M\},\#P_2)
        \le \La(n,\cP_6,\#P_2)
        \le (2+o(1))\nchn .
    \end{align*} \label{thm:P6}
\end{theorem}

In Theorem~\ref{thm:P5}, the extremal configuration is the set of all sets $F$ for which $|F\cap [n-2]|$ has size exactly $\lfloor n/2\rfloor-1$. 
However, in Theorem~\ref{thm:P6}, the  configuration that gives the lower bound is simply the largest antichain, together with $\{\emptyset,[n]\}$.  


The paper is organized as follows: 
In Section~\ref{sec:thm:Kst}, we prove Theorems~\ref{thm:K22} and~\ref{thm:Kst}. 
In Section~\ref{sec:P5P6}, we prove Theorems~\ref{thm:N},~\ref{thm:P5} and~\ref{thm:P6}. 
In Section~\ref{sec:open}, we present some open problems.

\section{Proofs of Theorems~\ref{thm:K22} and~\ref{thm:Kst}}
\label{sec:thm:Kst}

We start with the proof of Theorem~\ref{thm:Kst} because that proof uses the basic ideas that we use in both of the proofs. 
For Theorem~\ref{thm:K22}, some more careful details are necessary to ensure the precise result that is given.

\begin{proof}[Proof of Theorem \ref{thm:Kst}]
Let $\cF\subset 2^{[n]}$ be a $K_{s,t}$-free family. We color the sets of $2^{[n]}$ in the following way: 
If a set $G$ is strictly contained by (that is, ``below'') at least $t$ elements of $\cF$ then $G$ is blue, otherwise it is red. 
Note that the blue sets form a downset and the red sets form an upset. 
Furthermore, each blue set contains at most $s-1$ sets from $\cF$ (including, possibly, itself), otherwise a $K_{s,t}$ would be formed.
Define a \textit{critical pair} to be a pair of sets $(G,G')$ where $G\subsetneq G'$, $|G'|=|G|+1$, $G$ is blue and $G'$ is red. 

Let us count copies of $P_k$ of the form $F_1\subsetneq F_2 \subsetneq \dots \subsetneq F_k$, according to the colors of $F_1,F_k$. 
Suppose first that $F_k$ is blue. By definition, a blue set contains  at most $s-1$ other members of $\cF$,
  therefore the number of such $P_k$s is at most $\binom{s-1}{k-1}2^n$. 
Now assume that $F_1$ is red. 
By the definition of the coloring, a red set is below at most $t-1$ other members of $\cF$, hence the number of such $P_k$s is at most $\binom{t-1}{k-1}2^n$. 
This gives that the  number of monochromatic $P_k$s is at most $\left[\binom{s-1}{k-1}+\binom{t-1}{k-1}\right]2^n<2^{\max\{s+t\}+n}=O\left(\sqrt{n}\cdot \nchn\right)$. 

Finally, suppose that $F_1$ is blue and $F_k$ is red.  For every $k$-chain of the form $F_1\subsetneq F_2 \subsetneq \dots \subsetneq F_k$, there is a pair $F_1\subseteq G\subsetneq G'\subseteq F_k$ such that $(G,G')$ is a critical pair and each of $G$ and $G'$ are comparable with each set $F_i$, $i=1,\ldots,k$. Because the blue sets form a downset and the red sets form an upset, we have that   each  maximal chain $\cC$ contains at most one critical pair $(C,C')$ with  $C,C'\in\cC$. 

Now let us count the triples $(G,G',\cC)$ with $G,G'\in \cC$, the pair $(G,G')$ a critical pair, and the existence of $F_1\subsetneq F_2 \subsetneq \ldots \subsetneq F_k$ such that $F_1\subseteq G,G'\subseteq F_k$ and each $F_i$ is comparable with $G,G'$. 
Each such pair $(G,G')$ is contained in $|G|!(n-|G|-1)!$ maximal chains, and as we remarked earlier, every chain contains at most one such pair. 
So we obtain $\sum_{(G,G')}|G|!(n-|G|-1)!\le n!$. 
Since $k!(n-k-1)!$ is minimized when $k=\lceil\frac{n-1}{2}\rceil=\lfloor \frac{n}{2}\rfloor$, we obtain that the number of summands, and thus the number of pairs $(G,G')$ is at most
\begin{align*}
    \frac{n!}{\lfloor n/2\rfloor! \cdot \left(n-\lfloor n/2\rfloor-1\right)!} = \left\lceil\frac{n}{2}\right\rceil \cdot\nchn .
\end{align*} 

Finally, we claim that the number of $k$-chains in $\cF$ to which every pair $(G,G')$ belongs is at most $\sum_{i=1}^{k-1}\binom{s-1}{i}\binom{t}{k-i}$. 
Indeed, because $G$ is blue, it contains (allowing for itself) at most $s-1$ sets of $\cF$ and because $G'$ is red, it is contained (allowing for itself) in at most $t$ sets in $\cF$. 
So, $\binom{s-1}{i}\binom{t}{k-i}$ gives a bound on the number of $k$-chains such that the first $i$ sets of which are contained in $G$ and the remaining $k-i$ sets of which contain $G'$. 

Therefore, the number of $k$-chains with $F_1$ blue and $F_k$ red is at most 
$$\sum_{i=1}^{k-1}\binom{s-1}{i} \cdot \binom{t}{k-i} \cdot \left\lceil\frac{n}{2}  \right\rceil \cdot \nchn<2^{s+t-1}\cdot \left\lceil\frac{n}{2}\right\rceil \cdot \nchn.$$

So the total number of pairs in containment is at most $2^{\max\{s,t\}+n}+2^{s+t-1}\left\lceil\frac{n}{2}\right\rceil\nchn=O\left(\left\lceil\frac{n}{2}\right\rceil\nchn\right)$. 
This completes the proof of Theorem~\ref{thm:Kst}.
\end{proof}

\begin{remark}
    For the case $k=2$, it is easy to see that the proof gives 
    \begin{align*}\textstyle
        \La(n,K_{s,t},\#P_2)\leq \left(st-t+O\left(\frac{1}{\sqrt{n}}\right)\right)\cdot  \left\lceil \frac{n}{2} \right\rceil\cdot\binom{n}{\lfloor n/2\rfloor} .
    \end{align*}
    Though we will not provide the details here, the ideas in the proof of Theorem~\ref{thm:K22} can be used to  improve this to 
    \begin{align*}\textstyle
        \La(n,K_{s,t},\#P_2)\leq \left((s-1)(t-1)+O\left(\frac{1}{\sqrt{n}}\right)\right)\ \cdot \left\lceil \frac{n}{2}\right\rceil\cdot\binom{n}{\lfloor n/2\rfloor} .
    \end{align*}
\end{remark}~\\

We write $\bC_n$ to denote the set of all maximal chains in $[n]$ and for a chain $C_1,C_2,\dots,C_h$ we will write $\bC_{C_1,C_2,\dots,C_h}$ to denote all maximal chains in $\bC_n$ that contain each $C_i$, $i=1,\ldots,h$.

\begin{proof}[Proof of Theorem \ref{thm:K22}]
The set system $\binom{[n]}{\ceil{n/2}-1} \cup \binom{[n]}{\ceil{n/2}}$ is $\bfly$-free and has $\lceil\frac{n}{2}\rceil\cdot\nchn$ containments, which 
 proves the lower bound.

To prove the upper bound, consider a $\bfly$-free family $\cF\subset 2^{[n]}$. 
Our plan is as follows: 
We will partition the set of containments and associate to each part a disjoint subset of full chains. 
Then for each part in this partition, we will bound the ratio of the size of the associated subset of full chains to the size of that part.

Formally, suppose $\cP_1, \cP_2, \dots, \cP_m$ is a partition of $\cP=\{(F,F'): F\subsetneq F'\mbox{ for }F,F'\in \cF\}$ and there is a function $h: \{\cP_1,\cP_2,\dots,\cP_m\}\rightarrow 2^{\bC_n}$ such that $1\le i\neq j\le m$ implies $h(\cP_i)\cap h(\cP_j)=\emptyset$. 
Denote by $\alpha$  the minimum of $\frac{|h(\cP_i)|}{|\cP_i|}$ over $1\le i\le m$. 
By averaging, 
$$|\cP|=|\cP_1|+|\cP_2|+\ldots+|\cP_m|  \le \frac{\bC_n}{\alpha}=\frac{n!}{\alpha} $$
We  define a function $h$ with the above properties for which $\alpha$ is at least $\floor{ \frac{n}{2}}! \cdot \left(\ceil{\frac{n}{2}}-1\right)!$ because then
$$|\cP|\le \frac{n!}{\floor{\frac{n}{2}}! \cdot \left(\ceil{\frac{n}{2}}-1\right)!}=\ceil{ \frac{n}{2}}\cdot \nchn.$$

Denote $\cF_1$ (and $\cF_3$)  the set of inclusion-wise minimal (maximal) sets of $\cF$. 
Note that we may ignore any set in $\cF$ that is both minimal and maximal because such sets will not contribute to  the family of containments.
Let $\cF_2:= \cF-\cF_1-\cF_3$.  Note that $\cF_2$ is an antichain, otherwise we would have a $4$-chain, which forms a $\bfly$ poset. 

First consider all sets $Y_1,Y_2,\ldots,Y_s\in \cF_2$. 
Since $\cF$ is $\bfly$-free, for every $1\le i\le s$ there are two unique sets $X_i,Z_i\in\cF$ such that $X_i\subset Y_i\subset Z_i$, where $X_i\in\cF_1$ and $Z_i\in\cF_3$. Greedily set $\cP_i=\{(X_i,Y_i),~(Y_i,Z_i),~(X_i,Z_i)\}$, except if the pair $(X_i,Z_i)$ was already used for some $j<i$, then we set $\cP_i=\{(X_i,Y_i),~(Y_i,Z_i)\}$. 
In either case, we let $h(\cP_i)=\bC_{Y_i}$, the set of chains that contain $Y_i$. 
Clearly, $h(\cP_i)\ge\floor{\frac{n}{2}}!\cdot \ceil{\frac{n}{2}}!$ and $|\cP_i|\le 3$ and so 
 \begin{align}
    \frac{|h(\cP_i)|}{|\cP_i|}\ge \frac{1}{3} \cdot \floor{\frac{n}{2}}! \cdot \ceil{\frac{n}{2}}!\ge \floor{\frac{n}{2}}! \cdot \left(\ceil{\frac{n}{2}}-1\right)!  \qquad\mbox{ for every $1\le i\le s$.} \label{eq:3chain}
\end{align}
The last inequality comes from the condition that $n\ge 5$. \\

We still have to consider those pairs $F\subset F'$ for which $F\in\cF_1,~F'\in\cF_3$ and there is no set $Y\in\cF$ such that $F\subset Y\subset F'$. 
Let such pairs be $\left\{(F_1,F'_1),(F_2,F'_2),\dots,(F_t,F'_t)\right\}$. 
We set $\cP_{s+j}=\{(F_j,F'_j)\}$ and so $|\cP_{s+j}|=1$. 
In order to define $h(\cP_{s+j})$ for all $1\le j\le t$, we set 
\begin{align*}
    \cH_j   &=  \{H\in 2^{[n]}:F_j ~\text{is the unique subset (allowing for itself) of}\ H ~\text{in}\ \cF\}, \\
    \cH'_j  &=  \{H'\in 2^{[n]}:F'_j ~\text{is the unique superset (allowing for itself) of}\ H' ~\text{in}\ \cF\}. 
\end{align*}

Since $F_j\in\cF_1$ and $F_j'\in\cF_3$, we have $F_j\in\cH_j$ and $F_j'\in\cH'_j$. 
Furthermore, for any set $F_j\subset H \subset F_j'$, we have $H\in\cH_j\cup\cH'_j$, otherwise $H$ would have two proper subsets and two proper supersets in $\cF$, forming a $\bfly$. 
Here we used that there is no $H\in\cF$ for which  $F_j\subset H \subset F'_j$.

Observe that if $H\in\cH_j$, then every set between $F_j$ and $H$ is in $\cH_j$ and if $H'\in\cH'_j$, then every set between $H'$ and $F_j'$ is in $\cH'_j$. 
Thus, we can find two sets $G_j$ and $G_j'$ such that $F_j\subseteq G_j\subset G_j'\subseteq F_j'$, $G_j\in \cH_j$, $G'_j\in \cH_j'$ and $|G'_j|=|G_j|+1$. 
We let $h(\cP_{s+j})=\bC_{G_j,G_j'}$ for an arbitrary such pair, and thus
\begin{align}
    \frac{|h(\cP_{s+j})|}{|\cP_{s+j}|} = |h(\cP_{s+j})| = |G|! \cdot  (n-|G|-1)!
    \ge \ceil{\frac{n-1}{2}}! \cdot  \floor{\frac{n-1}{2}}! 
    = \floor{\frac{n}{2}}! \cdot  \left(\ceil{\frac{n}{2}}-1\right)! , \label{eq:2chain}
\end{align}
as claimed.

We still need to show that the images are pairwise disjoint. 
It is clear that $h(\cP_i)\cap h(\cP_j)=\bC_{Y_i}\cap \bC_{Y_j}=\emptyset$ where $1\le i,j\le s$, because $\cF_2$ is an antichain. 
Now consider $h(\cP_i)\cap h(\cP_{s+j})$ where $1\leq i\leq s$ and $1\leq j\leq t$. 
If a chain contains $G_j$ and $G_j'$, then it cannot contain any member of $\cF$ other than $F_j$ and $F'_j$ since $G_j \in\cH_j$ and $G_j'\in \cH'_j$. 
In particular such a chain cannot contain any $Y_i\in \cF_2$ and so $h(\cP_i)\cap h(\cP_{s+j})=\emptyset$.

Finally, suppose that a chain belongs to  $h(\cP_{s+a})\cap h(\cP_{s+b})$ via the pairs $(G_a,G_a')$ and $(G_b,G_b')$, respectively, where $1\le a,b\le t$. 
Let $S=G_a\cup G_b$.
Since both $G_a$ and $G_b$ are on the same chain, either $S=G_a$ or $S=G_b$.
Thus, $S$ is in both $\cH_a$ and $\cH_b$ and must contain only one member of $\cF$.
As a result, $F_a=F_b$. 
Analogously, using $S'=G_a'\cap G_b'$, we may conclude that $F_a'=F_b'$. 
Consequently, $a=b$ as desired. \\

It remains to establish equality in the case where $n\geq 7$. 
In order for equality to hold, both~\eqref{eq:3chain} and~\eqref{eq:2chain} must hold with equality for each $\cP_i$.
Since $n\geq 7$ implies~\eqref{eq:3chain} never holds with equality, $\cF_2$ is empty and $s=0$.
Moreover, in order for~\eqref{eq:2chain} to hold with equality, every pair $(G_j,G_j')$ must have sizes in $\left\{\lfloor\frac{n-1}{2}\rfloor, \lfloor\frac{n-1}{2}\rfloor+1\right\}$ or in $\left\{\lceil\frac{n-1}{2}\rceil, \lceil\frac{n-1}{2}\rceil+1\right\}$.

Next observe that if $\cF$ is extremal, then for every pair $(F_j,F'_j)$ we must have $F_j=G_j,F'_j=G'_j$. Indeed, if $F_j\subsetneq G_j\subsetneq G'_j\subseteq F'_j$, then there exists at least one other pair $(G^*_j,G^{**}_j)\neq (G_j,G'_j)$ with $F_j\subseteq G^*_j\subseteq G^{**}_j\subseteq F'_j$, $G^*_j\in \cH_j$, $G^{**}_j\in \cH_j'$ and $|G^*_j|+1=|G^{**}_j|$. So the mapping $h$ could be defined as $h(F_j,F'_j)=\bC_{G_j,G'_j}\cup \bC_{G^*_j,G^{**}_j}$, and then (\ref{eq:2chain}) could not hold with equality.

So $\cF_2$ is empty, for every pair $(F_j,F'_j)$ we have $F_j=G_j,F'_j=G'_j$, and the sizes are in $\left\{\lfloor\frac{n-1}{2}\rfloor, \lfloor\frac{n-1}{2}\rfloor+1\right\}$ or in $\left\{\lceil\frac{n-1}{2}\rceil, \lceil\frac{n-1}{2}\rceil+1\right\}$.

If $n$ is odd, then equality can only occur if $\cF=\binom{[n]}{(n-1)/2}\cup \binom{[n]}{(n+1)/2}$, as needed.

If $n$ is even, then each pair $(G,G')$ must contain an element from $\binom{[n]}{n/2}$ and an element from $\binom{[n]}{n/2-1}\cup\binom{[n]}{n/2+1}$. 
Because $\cF_2=\emptyset$, then no member of $\binom{[n]}{n/2}$ can have a member of $\cF$ both above and below it.
So in order for equality to hold, each $F\in\binom{[n]}{n/2}$ must be in $\cF$ and either $F$ has $n/2$ members of $\cF$ above it (an \textit{under-element}) or $n/2$ members of $\cF$ below it (an \textit{over-element}).
If there are both under elements and over elements, there must be a pair $F_1,F_2$ whose symmetric difference is 2, $F_1$ is an under-element and $F_2$ is an over-element. 
In that case, $F_1\cup F_2,F_1\cap F_2\in \cF$, which would put $F_1,F_2\in\cF_2$, a contradiction to that set being empty. 
Thus, in order for equality to hold, either $\cF=\binom{[n]}{\lfloor n/2\rfloor}\cup \binom{[n]}{\lfloor n/2\rfloor+1}$ or $\cF=\binom{[n]}{\lceil n/2\rceil-1}\cup \binom{[n]}{\lceil n/2\rceil}$. \\

This completes the proof of Theorem~\ref{thm:K22}.
%
\end{proof}


\section{Proofs of Theorems~\ref{thm:N},~\ref{thm:P5} and~\ref{thm:P6}}
\label{sec:P5P6}

\begin{defn}
Let $\cF$ be a set system. The \textit{comparability graph} of $\cF$ is a simple graph $G$ whose vertices correspond to the sets of $\cF$. The vertices representing two sets $A$ and $B$ are connected by an edge if $A\subset B$ or $B\subset A$ holds.
\end{defn}

\begin{defn}
The convex hull of a set system $\cF\subset 2^{[n]}$ is
$$conv(\cF)=\{F\in 2^{[n]} ~ \big| ~ \exists F_1,F_2\in \cF, ~ F_1\subseteq F\subseteq F_2\}.$$
\end{defn}

\begin{defn}
Two families $\cH$ and $\cH'$ in $2^{[n]}$ are \textit{incomparable} if neither $F\subseteq F'$ nor $F'\subseteq F$ holds for every pair $F\in \cH$, $F'\in \cH'$.
\end{defn}

Note that if $\cH$ and $\cH'$ are incomparable then $conv(\cH)$ and $conv(\cH')$ are also incomparable. Therefore no full chain can intersect both $conv(\cH)$ and $conv(\cH')$.

\begin{lemma}\label{sublattice}
Let $A\subset B\subseteq [n]$. The number of full chains meeting the set $\{F~:~A\subset F\subset B\}$ is $\frac{n!}{\binom{n-|B|+|A|}{|A|}}$.
\end{lemma}

\begin{proof}
A full chain meets the set $\{F~:~A\subset F\subset B\}$ if and only if all elements of $A$ appear before the elements of $[n]\setminus B$. This is true for a random chain with probability $\binom{n-|B|+|A|}{|A|}^{-1}$.
\end{proof}

\begin{proof}[Proof of Theorem~\ref{thm:N}]
To show $\La(n,N,\#P_2)\ge\nchn$, consider the family $\{\emptyset\}\cup{\binom{[n]}{\floor{n/2}}}$. It has $\nchn$ containments and avoids the subposet $\bN$.

To prove the upper bound, consider an $\bN$-free family $\cF\subset 2^{[n]}$, and let $G$ be its comparability graph. We will classify the possible  components of $G$.

If $\cF$ contains the poset $P_3$ (three sets $A\subset B\subset C$), then there can be no fourth set comparable to any of these three, otherwise an $\bN$ would be formed. A  component corresponding to a  $P_3$ is a triangle.

If a  component contains no $P_3$, then all its paths must be alternating. 
Since we cannot have $\bN$ (the $3$-edge alternating path), there is no path with at least three edges. 
Therefore, a component consists of either a minimal element and some of its supersets forming an antichain or a maximal element and some its subsets forming an antichain. 
In either case, the corresponding component in $G$ is a star.

Note that  a full chain can meet the convex hull of at most one of the components. We will show that if a component has $r$ containments (edges in $G$), then at least $r\cdot \floor{\frac{n}{2}}!\cdot \ceil{\frac{n}{2}}!$ full chains meet its convex hull.  Since there are $n!$ full chains, this means that the number of containments is at most $n!/\left(\floor{\frac{n}{2}}!\cdot \ceil{\frac{n}{2}}!\right)=\nchn$.

If a $P_3$ component is formed by the sets $A\subset B\subset C$ then its convex hull is the set $\{F~:~A\subset F\subset C\}$. By Lemma \ref{sublattice}, it meets  $\frac{n!}{\binom{n-|C|+|A|}{|A|}}$ full chains. Since $|C|-|A|\ge 2$, this is at least $\frac{n!}{\binom{n-2}{\floor{n/2}-1}}=n(n-1)\cdot\left(\left\lfloor\frac{n}{2}\right\rfloor-1\right)!\cdot\left(\left\lceil\frac{n}{2}\right\rceil-1\right)!$, which is at least $3\cdot \floor{\frac{n}{2}}!\cdot \ceil{\frac{n}{2}}!$ when $n\ge 3$.

A $P_3$-free component of size $r+1$ has $r$ containments. It also has $r$ sets forming an antichain. Each of them meets at least $\floor{\frac{n}{2}}!\cdot \ceil{\frac{n}{2}}!$ full chains, and these are different  due to the antichain property.
\end{proof}

The following simple lemma will be used to prove Theorem~\ref{thm:P5}. 
See also \cite{DG2015} which has the same main result about incomparable copies of a poset, proved independently.

\begin{lemma}[Katona, Nagy \cite{KN2015}]\label{chaincount}
Let $\cH$ be a family of $t$ subsets of $[n]$. Then the number of full chains meeting at least one element of $\cH$ is at least
$$\left( t-\frac{t(t-1)}{n} \right) \cdot\floor{\frac{n}{2}}! \cdot \ceil{\frac{n}{2}}!.$$
\end{lemma}

\begin{proof}[Proof of Theorem~\ref{thm:P5}]
For $n\ge 2$, consider the following family:
$$\cF:=\left\{F\subset [n] ~:~ |F\cap [n-2]|=\left\lfloor\frac{n-2}{2}\right\rfloor\right\}.$$
This family can be divided into $\binom{n-2}{\floor{n/2}-1}$ pairwise incomparable $4$-tuples. Each of these are isomorphic to the Boolean lattice $B_2$, and has $5$ containments. Therefore the  number of containments in $\cF$ is $5\binom{n-2}{\floor{n/2}-1}$.

To prove that for $n$  large enough  a $\cP_5$-free family $\cF$ has at most $5\binom{n-2}{\floor{n/2}-1}$ containments, we use the same strategy as in the proof of Theorem~\ref{thm:N}. We will describe all possible components of the comparability graph $G$, and show that if a component has $c$ containments then its convex hull meets at least
$$\frac{c \cdot n!}{5\cdot \binom{n-2}{\floor{n/2}-1}}=\left(\frac{4}{5}+o(1)\right)\cdot c \cdot \floor{\frac{n}{2}}! \cdot \ceil{\frac{n}{2}}!$$
full chains. There are $n!$ full chains and each of them  meets the convex hull of at most one component, therefore the number of containments is at most $5 \binom{n-2}{\floor{n/2}-1}$.

Before finding all possible components, we describe two types for which the statement is easy to verify. These two types will cover most of our cases.

If a component has at most $c\le 100$ containments, and its convex hull contains at least $c$ elements, then we call it \textit{type I}. Choose $c$ sets from the convex hull. By Lemma \ref{chaincount}, the number of full chains meeting them is at least $(c-o(1))\cdot \floor{n/2}!\cdot \ceil{n/2}!$, which is more than $\left(\frac{4}{5}+o(1)\right)\cdot c \cdot \floor{n/2}! \cdot \ceil{n/2}!$ for large enough $n$.

If a component has $c$ containments (for $c$ of any size) and at least $\frac{5}{6}c$ of its members form an antichain then we call it \textit{type II}. Each of the sets in the antichain meets at least $\floor{n/2}!\cdot\ceil{n/2}!$ full chains, and these are pairwise different. Therefore, a type II  component meets at least $\frac{5}{6}c\cdot \floor{n/2}!\cdot \ceil{n/2}!$ chains, which is more than $\left(\frac{4}{5}+o(1)\right)\cdot c\cdot \floor{n/2}! \cdot \ceil{n/2}!$ for large enough $n$.

Now let us list all possible  components of $G$ and the corresponding set systems in $\cF$. Note that $\cF$ being $\cP_5$-free means that $G$ has no $5$-vertex path.

First,  consider a tree component in $G$. As we work with $\cP_5$-free posets, it contains  no $4$-edge path.
 It is easy to see that this implies that there must be an edge that shares a vertex with all other edges. Let $A\subset B$ be the sets in $\cF$ corresponding to the endpoints of this special edge. If a third set $C\in\cF$ would be a subset of $A$ or a superset of $B$ then $A,B$ and $C$ would span a triangle in $G$, which is not allowed in this case. Therefore all other sets of $\cF$ are either supersets of $A$ or subsets of $B$, so the component is a star. If this component has $t$ sets, then the number of containments is $t-1$ and there is an antichain of size $t-2$, formed by all sets of the component, except for $A$ and $B$. If $t<7$ then the component is of type I. If $t\ge 7$ then it is of type II, since $t-2\ge\frac{5}{6}(t-1).$

From now on, we  assume that the component has a cycle. It cannot be a cycle of length $5$ or more since it would contain a $5$-vertex path. First, assume that the component has a triangle, but no $4$-cycle. A triangle in $G$ corresponds to three sets $A\subset B\subset C$ in $\cF$. If there are no other sets in this component then it is of type I. 

If there is a fourth set $D\in\cF$ such that $B\subset D$ or  $D\subset B$, then $\{A,C,B,D\}$ is a $4$-cycle, which is not allowed in this case. Therefore,  any set of $\cF$ comparable to $\{A,B,C\}$ must be either a superset of $A$ or a subset of $C$. If there are sets $D,E\in\cF$ such that $A\subset D$ and $E\subset C$, then $D,A,B,C,E$ form an $S\in \cP_5$. Note that $D=E$ would create a $4$-cycle, which is not allowed in this case.

We have only two subcases remaining to handle. Either there are some sets $D_1, D_2,\dots, D_k\in \cF$ such that $A\subset D_i$ for all $i$ and $\{D_1, D_2,\dots, D_k\}$ is incomparable to $\{B,C\}$, or there are some sets $E_1, E_2,\dots, E_k\in \cF$ such that $E_j\subset C$ for all $j$ and $\{E_1, E_2,\dots, E_k\}$ is incomparable to $\{A,B\}$. By symmetry, it is enough to consider the first subcase. The sets $\{D_1, D_2,\dots, D_k\}$ must form an antichain, since $D_i\subset D_j$ would mean that $\{B,C,A,D_j,D_i\}$ form an $M\in \cP_5$. Therefore the component has $k+3$ sets, $k+3$ containments, and contains the antichain $\{B,D_1, D_2,\dots, D_k\}$ of size $k+1$. If $k<9$ then the component is of type I. If $k\ge 9$ then it is of type II,  since $k+1\ge \frac{5}{6}(k+3)$.

Finally, we handle the case when  a component contains  a $4$-cycle. If there were a fifth set in $\cF$ comparable to any of the four sets in the cycle in any way then these sets would form a poset in $\cP_5$. Therefore, we can assume that this component consists of  only four sets. If there are only four containments, then the component is of type I. If there are six containments among the four sets then they must form a chain $A\subset B\subset C\subset D$. The convex hull is the set $\{F~:~A\subset F\subset D\}$. By Lemma \ref{sublattice}, it meets  $\frac{n!}{\binom{n-|D|+|A|}{|A|}}$ chains. Since $|D|-|A|\ge 3$, this is at least $\frac{n!}{\binom{n-3}{\floor{(n-3)/2}}}$, which is more than $\frac{6\cdot n!}{5\cdot\binom{n-2}{\floor{n/2}-1}}$ for $n\geq 3$.

The last subcase has five containments among the four sets. This can be achieved in three different ways:
\begin{itemize}
    \item $A\subset B\subset C,D$. In this case, the size of the convex hull is at least five since it also contains at least one set $E$ such that $E\not=B$ and $A\subset E\subset D$. With five containments and a convex hull of size at least five, the component is of type I.
    \item $A,B\subset C\subset D$. This case can be handled similarly to the previous one.
    \item $A\subset B,C\subset D$. The convex hull of this component is the set $\{F~:~A\subset F\subset D\}$.
    By Lemma \ref{sublattice}, it meets $\frac{n!}{\binom{n-|D|+|A|}{|A|}}$ full chains. Since $|D|-|A|\ge 2$ must hold, the minimum of this quantity is $\frac{n!}{\binom{n-2}{\floor{n/2}-1}}=\frac{5\cdot n!}{5\cdot\binom{n-2}{\floor{n/2}-1}}$.
\end{itemize}

Thus, every component with $c$ containments has a convex hull that meets at least $\frac{c \cdot n!}{5\cdot \binom{n-2}{\floor{n/2}-1}}$ full chains, as desired.

\end{proof}

The following important result of Bukh~\cite{Bukh} will be used for proving Theorem~\ref{thm:P6}.

\begin{theorem}[Bukh \cite{Bukh}]\label{bukh}
    For every poset $T$ with a tree Hasse diagram, we have $La(n,T)=(h(T)-1+o(1))\cdot \binom{n}{\lfloor n/2\rfloor}$, where $h(T)$ denotes the height of $T$.
\end{theorem}


\begin{proof}[Proof of Theorem~\ref{thm:P6}]
To prove the first inequality, let $\cF=\{\emptyset\}\cup{\binom{[n]}{\floor{n/2}}}\cup \{[n]\}$. This family contains neither $M$ nor $W$. The number of containments in $\cF$ is $2\nchn+1$.

To prove the second inequality, assume that we have six pairwise different sets $A_1, A_2,\dots, A_6$ such that $A_i\subset A_{i+1}$ or $A_i\supset A_{i+1}$ holds for all $1\le i \le 5$. It can be checked that one can select five of these sets forming an $M$ or a $W$. Indeed, let $m$ be the most number of consecutive $A_i$s that form a chain, so $A_i\subset A_{i+1}\subset  \ldots \subset A_{i+m-1}$ or $A_i\supset A_{i+1}\supset \ldots \supset A_{i+m-1}$.
 If $m\ge 5$, then a $5$-chain is both an $M$ and a $W$. 
 If $m=2$, then one of $\{A_1,A_2,\dots, A_5\}$ and $\{A_2,A_3,\dots,A_6\}$ forms an $M$ and the other forms a $W$. If $m=4$, then by symmetry we can assume $A_i \subset A_{i+1}\subset A_{i+2}\subset A_{i+3}$. If $i\ge 2$, then $\{A_{i-1},A_i,A_{i+2},A_{i+1},A_{i+3}\}$ forms a $W$, otherwise $\{A_i,A_{i+2},A_{i+1},A_{i+3},A_{i+4}\}$ forms an $M$. Finally, if $m=3$, then again we can assume $A_i \subset A_{i+1}\subset A_{i+2}$. 
If $i\le 2$ then we have $A_{i+2}\supset A_{i+3}$ by the maximality of the chain $A_i \subset A_{i+1}\subset A_{i+2}$. If $A_{i+3}\subset A_{i+4}$ then $\{A_{i+1},A_{i},A_{i+2},A_{i+3},A_{i+4}\}$ forms a $W$. If $A_{i+3}\supset A_{i+4}$ then $\{A_{i+1},A_{i},A_{i+2},A_{i+4},A_{i+3}\}$ forms a $W$. If $i\ge 3$ then we can similarly conclude that $A_{i-1}\supset A_i$ and either $\{A_{i-2},A_{i-1},A_{i},A_{i+2},A_{i+1}\}$ or $\{A_{i-1},A_{i-2},A_{i},A_{i+2},A_{i+1}\}$ forms an $M$.
This means that a family that avoids both $M$ and $W$ will avoid all posets in $\cP_6$, proving the inequality.

To prove the last relation, assume that $\cF \subset 2^{[n]}$ is a $\cP_6$-free family. It means that its comparability graph $G$ contains no $6$-vertex path. A theorem of Erd\H{o}s and Gallai~\cite{EG} states that in this case $|E(G)|\le 2|V(G)|$. Therefore the number of containments in $|\cF|$ is at most $2|\cF|$. Since $\cF$ avoids the $6$-element alternating path poset (a tree poset of height two), Theorem \ref{bukh} implies $|\cF|\le (1+o(1))\nchn$, completing the proof. 
\end{proof}

\section{Open problems}
\label{sec:open}

Determining the exact or asymptotic value of any $\La(n,P,\#Q)$ would be interesting, but there are two natural problems that arise from our current results and previous findings from \cite{GKP} and \cite{GMNPV}.

\begin{problem}
    Determine the largest possible order of magnitude of $\La(n,P,\#P_k)$ over all posets of height at most $\ell\le k$. 
    The case $\ell=2$ is settled for arbitrary values of $k$: \tref{Kst} gives the upper bound $La(n,P,\#P_2)=O_k\left(n\cdot \binom{n}{\lfloor n/2\rfloor}\right)$, while the family 
    $$\cF_k:=\left\{F\in 2^{[n]}:\left|F\cap [n-k+2]\right|=\binom{n}{\lfloor n/2\rfloor} ~\text{or}\ \left|F\cap [n-k+2]\right|=\binom{n}{\lfloor n/2\rfloor}+1\right\}$$ 
    is $K_{s,s}$-free for large enough $s$ and contains $\Omega_k\left(n\cdot \binom{n}{\lfloor n/2\rfloor}\right)$ chains of length $k$.
\end{problem}

Let us repeat some observations on posets of height $2$ and chains of size $2$. It is clear that if the family consisting of the two middle levels does not contain any copy of $P$, then we have $\La(n,P,\#P_2)\ge \lceil \frac{n}{2}\rceil\cdot \nchn$. As mentioned in the introduction, $\La(n,T,\#P_2)=O_T\left(\nchn\right)$ was proved in~\cite{GMNPV}.

\begin{problem}
Does there exist a poset $P$ contained in $\binom{[n]}{\lfloor n/2\rfloor}\cup \binom{[n]}{\lfloor n/2\rfloor+1}$ such that $\La(n,P,\#P_2)=\omega\left(\nchn\right)$?
\end{problem}

Observe that $\bfly$ is a poset for which the value of $\La(n,\bfly,\#P_k)$ is determined for all values of $k$. The ordinary forbidden subposet problem, case $k=1$, was solved by DeBonis, Katona, and Swanepoel~\cite{DKS}. The case $k=2$ is settled by Theorem~\ref{thm:K22}, the case $k=3$ is an easy proposition by Gerbner, Keszegh, and Patk\'os~\cite{GKP}, while for all $k\ge 4$ the value is zero as a 4-chain contains $\bfly$. It would be interesting to find other posets $P$ for which all values $La(n,P,\#P_k)$ can be determined. 

\begin{problem}
Determine $\La(n,\Diamond_3,\#P_k)$ for $k=2,3,4$, where $\Diamond_3$ is the poset on 5 elements $a,b_1,b_2,b_3,c$ with $a<b_i<c$ for all $i=1,2,3$. The case $k=1$ was solved by Griggs, Li, and Lu~\cite{GLL}.
\end{problem}


\begin{thebibliography}{99}

\bibitem{AS}
\textsc{N. Alon, C. Shikhelman}. Many $T$ copies in $H$-free graphs. \textit{J. Combin. Theory Ser. B} \textbf{121} (2016), 146--172.


\bibitem{Bukh}
\textsc{B. Bukh.} Set families with a forbidden subposet, \emph{Electron. J. Combin.}, \textbf{16} (2009), R142, 11pp.

\bibitem{DKS}       
\textsc{A. De Bonis, G.O.H. Katona, K.J. Swanepoel}. Largest family without $A\cup B\subseteq C\cap D$. \textit{J. Combin. Theory Ser. A}, \textbf{111}(2) (2005), 331--336. 

\bibitem{DG2015}
\textsc{A.P. Dove, J.R. Griggs}. Packing posets in the Boolean lattice. \textit{Order}, \textbf{32}(3) (2015), 429-438.

\bibitem{E1945} \textsc{P. Erd\H os}. On a lemma of Littlewood and Offord. \textit{Bulletin of the American Mathematical Society}, \textbf{51}(12), 898--902, 1945.

\bibitem{EG}
\textsc{P. Erd\H os, T. Gallai}. On maximal paths and circuits of graphs. \textit{Acta Math. Acad. Sci. Hungar.} \textit{10} (1959) 337--356.

\bibitem{ES1946}
\textsc{P. Erd\H os, A.H. Stone}. On the structure of linear graphs. \textit{Bull. Amer. Math. Soc.} \textbf{52} (1946), 1087--1091.

\bibitem{ES1966}
\textsc{P. Erd\H os, M. Simonovits}. A limit theorem in graph theory. \textit{Studia Sci. Math. Hungar.} \textbf{1} (1966), 51--57.

\bibitem{GKP}
\textsc{D. Gerbner, B. Keszegh, B. Patk\'os.} Generalized forbidden subposet problems, \textit{Order}, \textbf{37}(2) (2020) 389--410.

\bibitem{GMNPV}
\textsc{D. Gerbner, A. Methuku, D.T. Nagy, B. Patk\'os, M. Vizer},  On the number of containments in P-free families. \textit{Graphs Combin.}, \textbf{35}(6) (2019), 1519--1540.

\bibitem{GP2008} \textsc{D. Gerbner, B. Patk\'os.} $l$-chain profile vectors. \textit{SIAM J. Discrete Math.}, \textbf{22}(1), (2008) 185--193.

\bibitem{GP2019} \textsc{D. Gerbner, B. Patk\'os.} \textit{Extremal Finite Set Theory}, CRC Press, Boca Raton, FL, 2019.

\bibitem{GLi}
\textsc{J.R. Griggs, W-T. Li}. Progress on poset-free families of subsets. \textit{Recent trends in combinatorics}. Springer, Cham. (2016), 317--338.

\bibitem{GLL}
\textsc{J.R. Griggs, W-T. Li, L. Lu}.  Diamond-free families. \textit{Journal of Combinatorial Theory, Series A}, \textbf{119}(2) (2012), 310--322.

\bibitem{GL}
\textsc{J.R. Griggs, L. Lu}. On families of subsets with a forbidden subposet. \textit{Combin. Probab. Comput.}, \textbf{18}(5) (2009), 731--748.

\bibitem{GMT}
\textsc{D. Gr\'osz, A. Methuku, C. Tompkins}. An upper bound on the size of diamond-free families of sets. \textit{J.~Combin. Theory Ser. A}, \textbf{156} (2018), 164--194.

\bibitem{K1973} \textsc{G.O.H. Katona}. Two applications of Sperner type theorems (for search theory and truth functions). \textit{Period. Math. Hungar.}, \textbf{3} (1973), 19--26.

\bibitem{KN2015}
\textsc{G.O.H. Katona, D.T. Nagy}.  Incomparable copies of a poset in the Boolean lattice. \textit{Order}, \textbf{32}(3) (2015), 419--427.

\bibitem{KT}
\textsc{G.O.H. Katona, T. Tarj\'an}. Extremal problems with excluded subgraphs in the $n$-cube, Graph Theory, Lagow, 1981, Lecture Notes in Math. 1018 (Springer-Verlag, Berlin, 1983) 84--93.

\bibitem{P2009} 
\textsc{B. Patk\'os}. The distance of $\cF$-free hypergraphs. \textit{Studia Scientiarum Mathematicarum Hungarica}, \textbf{46}(2) (2009), 275--286.

\bibitem{S1928} \textsc{E. Sperner}. Ein Satz \"uber Untermengen einer endlichen Menge. \textit{Mathematische Zeitschrift}, \textbf{27}(1) (1928), 544--548.

\bibitem{T1941} \textsc{P. Tur\'an}. Eine Extremalaufgabe aus der Graphentheorie. (Hungarian) \textit{Mat. Fiz. Lapok} \textbf{48} (1941), 436–452.

\end{thebibliography}
\end{document}